\documentclass{amsart}

\usepackage{wasysym}
\usepackage{amsmath}
\usepackage{amssymb}
\usepackage{amsthm}
\usepackage{tikz}
\usetikzlibrary{matrix,arrows,backgrounds,shapes.misc,shapes.geometric,patterns,calc,positioning}
\usetikzlibrary{calc,shapes}
\usepackage{wrapfig}
\usepackage{epsfig}  
\usepackage{color}	 %
\input{xy}
\xyoption{poly}
\xyoption{2cell}
\xyoption{all}

\newtheorem{thm}{Theorem}[section]
\newtheorem{prop}[thm]{Proposition}
\newtheorem{lem}[thm]{Lemma}
\newtheorem{cor}[thm]{Corollary}

\theoremstyle{definition}
\newtheorem{definition}[thm]{Definition}

\theoremstyle{remark}

\newtheorem{remark}[thm]{Remark}
\newtheorem{example}[thm]{Example}

\numberwithin{equation}{section}

\newcommand{\za}{\alpha}
\newcommand{\zb}{\beta}
\newcommand{\zd}{\delta}

\newcommand{\ze}{\epsilon}
\newcommand{\zg}{\gamma}

\newcommand{\zl}{\lambda}
\newcommand{\zs}{\sigma}

\newcommand{\zO}{\Omega}

\newcommand{\kk}{\Bbbk}

\newcommand{\rad}{\textup{rad}\,}

\newcommand{\id}{\textup{id}}
\newcommand{\pd}{\textup{pd}}
\newcommand{\Supp}{\textup{Supp\,}}
\newcommand{\soc}{\textup{soc}}
\newcommand{\Ext}{\textup{Ext}_A}

\title{A note on sequential walks}
\author{Ibrahim Assem}
\address{
D\'epartement de Math\'ematiques, Universit\'e de Sherbrooke,  Sherbrooke, Qu\'ebec, Canada, J1K 2R1}
\email{ibrahim.assem@usherbrooke.ca} 

\author{Mar\'\i a Julia Redondo}
\address{Departamento de Matem\'atica, Instituto de Matem\'atica,
Universidad Nacional del Sur, Av. Alem 1253, 8000 Bah\'\i a Blanca,
Argentina}
\email{mredondo@criba.edu.ar}
\author{Ralf Schiffler}\thanks{The first author was supported by  NSERC of Canada.
The second author is a researcher from CONICET of Argentina. 
The third author was supported by the NSF-grant  DMS-1800860, and by the University of Connecticut.}
\address{Department of Mathematics, University of Connecticut, 
Storrs, CT 06269-3009, USA}
\email{schiffler@math.uconn.edu}
\dedicatory{Dedicated to Jos\'e Antonio de la Pe\~na on the occasion of his sixtieth birthday}
\begin{document}
\begin{abstract}
 This short note is devoted to motivate and clarify the notion of sequential walk introduced by the authors in \cite{ARS}. We also give some applications of this concept.
\end{abstract}
\maketitle
\section{Introduction}
The class of tilted algebras, introduced by Happel and Ringel \cite{HR}, is among the most ubiquitous classes in the representation theory of algebras. For instance, any cluster-tilted algebra is the trivial extension of a tilted algebra by a particular bimodule \cite{ABS}. Surprisingly enough, it is difficult to check whether a given algebra is tilted or not without a good knowledge of its module category. 
Indeed, most known criteria revolve around the existence of a combinatorial configuration called a complete slice, see, for instance, \cite{AsSiSk}.

It was thus needed to have a handy criterion depending only on the bound quiver of the algebra. The most powerful criterion so far is the existence of so-called sequential walks. Sequential walks have a long history; they first appeared in \cite{A}, where it was shown that an iterated tilted algebra of type $\mathbb{A}$ is tilted if and only if it has no sequential walk. They surfaced again in \cite{HL} under the name of  ``sequential pairs'' in the classification of quasi-tilted string algebras, then in \cite{BT} in the classification of shod string algebras and in \cite{Dionne}  in the classification of laura string algebras under the name of ``double zeros''. Their present guise
was introduced in \cite{ARS} in the context of non-necessarily monomial algebras. It was proved there that if an algebra contains a sequential walk, then it is not tilted. 

This shows that this notion is very natural. Indeed, as we prove here, it follows from simple considerations on the comparison between the shape of the bound quiver and the homological dimensions of some modules.

It was pointed out to the authors that the definition of sequential walk given in \cite{ARS} contained an ambiguity. It is the first purpose of this note to clarify this ambiguity. While doing so, we slightly generalize the definition of sequential walk, and try to illustrate its usefulness for computing homological dimensions.

Our main result is the following.

\begin{thm}
 Let $A$ be a finite dimensional algebra over an algebraically closed field. If the quiver of $A$ contains a sequential walk, then $A$ is not shod. In particular, $A$ is neither quasi-tilted nor tilted.
\end{thm}

This note is organized as follows.
 In section \ref{sect 2}, we recall the definitions and known results which are necessary for the proof of our theorem. In  section \ref{sect 3}, we give our  definition of sequential walk, trying to motivate it, then we prove our theorem. Section \ref{sect 4} is then devoted to applications and examples. 
 
The authors thank Claire Amiot for useful discussions on the topic.
\section{Preliminaries}\label{sect 2}
\subsection{Notation}\label{sect 1.1}
Throughout, we let $\kk$ denote an algebraically closed field
and $A$ a finite-dimensional basic $\kk$-algebra.
We recall that any basic and connected finite dimensional $\kk$-algebra $A$ can be written in the form $A\cong \kk Q_A/I$ where $\kk Q_A$ is the path algebra of the quiver $Q_A$ of $A$ and $I$ is an ideal generated by finitely many relations. The pair $(Q_A,I)$ is then  called a \emph{bound quiver}.
A \emph{relation} on  a quiver is a linear combination $\rho=\sum_{i=1}^m \zl_iu_i$, where the $\zl_i$ are nonzero scalars and the $u_i$ are paths of length at least two having all the same source and the same target.  It is called a \emph{monomial relation} if it is a path, and a \emph{minimal relation} if $m\ge 2 $ and, for any nonempty proper subset $J\subset \{1,2,\ldots, m\}$ one has $\sum_{j\in J} \zl_j u_j\notin I$.

 Relations in a bound quiver $(Q,I)$ are generated by  top relations. Indeed, let $\kk Q^+$ be the two-sided ideal of $\kk Q$ consisting of the linear combinations of paths of length at least one and $e_x$ be the primitive idempotent of $A$ corresponding to a point $x\in(Q_A)_0$. Then a relation $\rho\in e_xI\,e_y$ is called a \emph{top relation} if its residual class 
 in $e_x\left(\frac{I}{\kk Q^+ I+I\kk Q^+}\right) e_y$ is nonzero.  They are called top relations because they correspond to nonzero elements of the top of $I$, considered as a $\kk Q$-$\kk Q$-bimodule.
 Intuitively one may think of the top relations as being the shortest ones.

 For a point $x$ in the ordinary quiver of $A$, we denote by $P_x$, $I_x$, $S_x$ respectively, the indecomposable projective, injective and simple $A$-modules corresponding to $x$.   The \emph{support} of an $A$-module $M$ is the full subquiver $\Supp M$ generated by the points $x\in (Q_A)_0$ such that $Me_x\ne 0$. For further definitions and facts, we refer the reader to \cite{AsSiSk, S}.

 \subsection{Classes of algebras}\label{sect 1.2}
 We need the following classes of algebras.

\begin{definition}
 (a) \cite{CL} An algebra $A$ is called \emph{shod} (for \underline small \underline {ho}mological \underline dimensions) if every indecomposable $A$-module is of projective or injective dimension at most one.
 
 (b) \cite{HRS} An algebra $A$ is called \emph{quasi-tilted} if it is shod of global dimension at most two, or, equivalently, if it is the endomorphism algebra of a tilting object in a hereditary, locally finite abelian $\kk$-category. 
 
 (c) \cite{HR} An algebra is called \emph{tilted} if it is the endomorphism algebra of  a tilting module over a hereditary algebra. 
\end{definition}
In particular, every tilted algebra is quasi-tilted, and every quasi-tilted algebra is shod.

\subsection{Full subcategories}
There is a reduction procedure which we shall use in the proof of our main result, called taking full subcategories.

Let $e\in A$ be an idempotent. The finite dimensional $\kk$-algebra $eAe$ is called the \emph{full subcategory} determined by $e$; indeed, if one considers $A$ as a category whose objects are the elements of a complete set of primitive orthogonal idempotents and the morphism space from $e_x$ to $e_y$ is $e_x A \,e_y$, then a full subcategory is always of the aforementioned form with $e$ equal to a sum of primitive idempotents. 

We need the following lemma.

\begin{lem}
 \label{lem 1.2}
 Let $A$ be an algebra and $e\in A$ an idempotent, then
 \begin{itemize}
\item [{\rm (a)}] If $A$ is shod, then so is $e A e$. 
\item [{\rm (b)}] If $A$ is quasi-tilted, then so is $e A e$.
\item [{\rm (c)}] If $A$ is tilted, then so is $e A e$.
\end{itemize}
\end{lem}
\begin{proof}
 (a) is proved in \cite[1.2]{KSZ}, (b) in \cite[II.1.15]{HRS} and (c) in \cite[III.6.5]{H}.
\end{proof}

\subsection{Split-by-nilpotent extensions}
Let $A$ be an algebra and $E$ an $A$-$A$-bimodule which is finite dimensional as a $\kk$-vector space. We say that $E$ is equipped with an \emph{associative product} if there exists an $A$-$A$-bimodule morphism $E\otimes_A E\to E, e\otimes e'\mapsto ee'$, such that $e(e'e'')=(ee')e''$ for all $e,e',e''\in E$. 

\begin{definition}
 Let $A,E$ be as before. An algebra $R$ is called a \emph{split extension of $A$ by $E$} if 
 \[ R=\{(a,e)\mid a\in A, e\in E\}\]
 is equipped with the componentwise addition and the multiplication defined by
 \[(a,e)(a',e')=(aa',ae'+ea'+ee')\] 
 for $(a,e),(a',e')\in R$. If $E$ is  nilpotent with respect to its product, then $R$ is called a \emph{split-by-nilpotent extension}.
 \end{definition}
 It is clear that, if $R$ is a split extension of $A$ by $E$, then there exists an exact sequence
\[\xymatrix{0\ar[r]&E\ar[r]&R\ar[r]^\pi &A\ar[r]& 0},\]
where the projection $\pi$ is an algebra morphism having a section $\zs\colon A\to R$.  Also, $E$ is nilpotent if and only if it is contained in  $\rad R$.
If an exact sequence as above and a section $\zs$ to $\pi$  are given, then we say that this sequence {\em realizes} $R$ as a split extension of $A$ by $E$.
If $E^2=0$, then the split extension is called a \emph{trivial extension.}

 We need the following lemma.

\begin{lem}
 \label{lem 1.3}
 Let $R$ be a split extension of an algebra $A$ by a nilpotent bimodule $E$. 
 \begin{itemize}
\item [{\rm (a)}] If $R$ is shod, then so is $A $. 
\item [{\rm (b)}] If $R$ is quasi-tilted, then so is $A$.
\item [{\rm (c)}] If $R$ is tilted, then so is $ A$.
\end{itemize}
\end{lem}
\begin{proof}
 (a) and (b) are proved in \cite[2.5]{AZ2} and (c) in \cite{Z}.
\end{proof}

\subsection{Cutting arrows} \label{sect 1.5}
Let $R$ be a split extension of an algebra $A$ by a nilpotent bimodule $E$. 
We recall how one can pass from $R$ to $A$ by dropping arrows from the quiver of $R$. Let $w$ be a path in the quiver $Q_A$ of $A$ and $\za$ an arrow such that there exist subpaths $w_1,w_2$ of $w$ such that $w=w_1\za w_2$, then we write $\za| w$. Also, when we speak of a relation,  then we can assume without loss of generality that it is monomial or minimal.  Let thus $\rho =\sum_i \zl_iu_i$ be a relation on $Q_A$ from $x$ to $y$, say, with the $\zl_i$ nonzero scalars and the $u_i$ paths of length at least two from $x$ to $y$. We say that $\rho $ is \emph{consistently cut} by a set $S$ of arrows (or that $S$ is a {\em consistent cut}), if, whenever there exist $i$ and $\za_i| u_i$ such that $\za_i\in S$ then, for any $j\ne i$, there exists $\za_j| u_j$ satisfying $\za_j\in S$.

\begin{thm}
 \label{thm 1.5}
 \cite[2.5]{ACT} 
 Let $\eta_R\colon \kk Q_R\to R$ be a presentation of $R$, let $E$ be an ideal of $R$ generated by the classes modulo $I_R=\ker\eta_R$ of a set $S$ of arrows, and let $\pi\colon R\to R/E=A$ be the projection. Assume moreover that every relation in $I_R$ is consistently cut by $S$, then the exact sequence
 \[\xymatrix{0\ar[r]&E\ar[r]&R\ar[r]^\pi&A\ar[r]&0}\]
 realizes $R$ as a split extension of $A$ by $E$.
\end{thm}

\subsection{Reduction of an algebra}\label{sect 1.6}
We now define a notion  which we call reduction of an algebra.  We say that an algebra $B$ is a \emph{reduction} of an algebra $A$ if there exist an idempotent $e\in A$ and a two-sided ideal $E$ of $eAe$ contained in its radical such that the sequence
\[\xymatrix{0\ar[r]&E\ar[r]&eAe\ar[r]&B\ar[r]&0}\]
realizes $eAe$ as a split extension of $B$ by the nilpotent bimodule $E$. Thus $B$ is obtained from $A$ by the combination of two consecutive processes: one first drops points of $Q_A$ by passing to the full subcategory $eAe$, then one drops arrows as explained in Theorem \ref{thm 1.5}. One then has the following obvious corollary of lemmata \ref{lem 1.2} and \ref{lem 1.3}.
\begin{cor}
 \label{cor 1.6}
 Let $B$ be a reduction of $A$. 
 \begin{itemize}
\item [{\rm (a)}] If $A$ is shod, then so is $B$. 
\item [{\rm (b)}] If $A$ is quasi-tilted, then so is $B$.
\item [{\rm (c)}] If $A$ is tilted, then so is $ B$.
\end{itemize}

\end{cor}

\section{Sequential walks}\label{sect 3}
\subsection{The definition}
Given an arrow $\za$ in a quiver, its formal inverse is denoted by $\za^{-1}$, where we  agree that the source of $\za^{-1}$ is the target of $\za$ and the target of $\za^{-1}$ is the source of $\za$. 
A \emph{walk} $w$ is an expression of the form $w=\za_1^{\ze_1}\cdots \za_t^{\ze_t}$, where the $\za_i$ are arrows and the $\ze_i\in \{+1,-1\}$ are such that the target of $\za_i^{\ze_i}$ equals the source of $\za_{i+1}^{\ze_{i+1}}$, for all $i$. Such a walk is called \emph{reduced} if it contains no expression of one of the  forms $\za\za^{-1}$ or $\za^{-1}\za$, with $\za $ an arrow. It is called a \emph{zigzag walk} if it is of the form $\za_1^{\ze_1}\cdots \za_t^{\ze_t}$ with $\ze_i\ne \ze_{i+1}$ for all $i$.

We start by recalling the definition of sequential walk as stated in \cite{ARS}. Let $w$ be a nontrivial walk in a bound quiver $(Q,I)$. Assume that one writes $w=uw'v$ where each of $u,w',v$ is a subwalk of $w$. We say that $u,v$ \emph{point to the same direction} in $w$ if both $u$ and $v$, or both $u^{-1}$ and $v^{-1}$  are paths in $Q$.

A reduced walk $w=uw'v$ having $u,v$ pointing to the same direction was called a sequential walk in \cite{ARS} if there is a relation $\rho=\sum_i\zl_i u_i$ such that $u=u_1$, or $u=u_1^{-1}$,
there is a relation $\zs=\sum_i\mu_i v_i$ such that $v=v_1$, or $v=v_1^{-1}$
and no subpath $w_1$ of $w'$, or of $(w')^{-1}$ is involved in a relation of the form $\sum \nu_i w_i$. 

As mentioned in the introduction, there is an ambiguity in this definition arising from the undefined word ``involved''. Indeed, the word ``involved'' can be understood as meaning ``is a branch of a relation''. But this is not correct as shown in the following example. 

\begin{example}\label{ex 1}
 Let $A$ be given by the quiver
 \[\xymatrix{1\ar@<2pt>[r]^{\za_1}\ar@<-2pt>[r]_{\zb_1}
 &2 \ar@<2pt>[r]^{\za_2}\ar@<-2pt>[r]_{\zb_2}
&3}\]
bound by the relation $\za_1\za_2+\zb_1\zb_2 =0$. Then $A$ is the one-point extension of the Kronecker algebra with quiver 
 \[\xymatrix{2 \ar@<2pt>[r]^{\za_2}\ar@<-2pt>[r]_{\zb_2}
&3}
\]
by the indecomposable postprojective module $\begin{smallmatrix} 2\ 2\\3\ 3\ 3\end{smallmatrix}$. In particular, $A$ is tilted \cite[3.5]{OS}. 

Consider the reduced walk $w=(\za_1\za_2)(\zb_2^{-1}\za_1^{-1})(\zb_1\zb_2)$  in the quiver of $A$, with $u=\za_1\za_2$, $v=\zb_1\zb_2$ and $w'=\zb_2^{-1}\za_1^{-1}$. Then neither $w'$ nor $(w')^{-1}=\za_1\zb_2$ is a branch of any relation while $u$ and $v$ satisfy the conditions of the definition of \cite{ARS}. If one understands ``involved'' as meaning ``is a branch of a relation'', then $w$ would be a sequential walk and  we would get a counterexample to \cite[2.4]{ARS}.
\end{example}

We propose the following definition.
\begin{definition}
 \label{main def}
 Let $w$ be a reduced walk in a bound quiver $(Q,I)$, then $w$ is called a {\em sequential walk} if the following hold.
 
 \begin{itemize}
\item [{\rm (a)}] $w=uw'v$, where $u,v$  point to the same direction, there is  a top relation  
$\rho=\sum_i\zl_i u_i$ such that $u=u_1$, or $u=u_1^{-1}$, and
there is a top relation $\zs=\sum_i\mu_i v_i$ such that $v=v_1$, or $v=v_1^{-1}$;
\item [{\rm (b)}] no subpath  of $w'$, or of $(w')^{-1}$ lies in $I$, nor is a branch of a relation having a branch which has a point in common with one of the $u_i$ or $v_i$;
\item [{\rm (c)}] $w'$ itself has no arrows in common with one of the $u_i $ or $v_i$.
\end{itemize}
\end{definition}

\begin{remark}
Condition (b)  holds for example if no subpath of $w'$ or $(w')^{-1}$ is the branch of a relation.
\end{remark}

\begin{remark}
 \label{rem 3} Our definition of sequential walk is clearly inspired by the  definition of sequential pair in \cite{HL} but it is not identical to it.
 Let $A=\kk Q/I$ be a string algebra. In particular $A$ is monomial. A sequential pair   of monomial relations  is a reduced walk $w$ that contains exactly two zero-relations and these two zero-relations point to the same direction in $w$.
 Our definition differs from \cite{HL} in the case of string algebras. For example, the algebra given by the quiver 
%
 
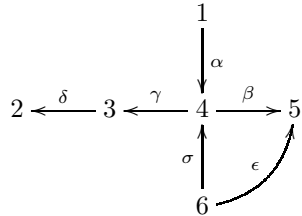
\[\xymatrix{
 &&1\ar[d]^\za\\
 2&3\ar[l]_\zd&4\ar[l]_\zg\ar[r]^\zb&5\\
 &&6\ar[u]^\zs\ar@/_{10pt}/[ru]^\ze
 }\]
 bound by the relations $\za\zb=\zg\zd=\zs\zb=0$ contains
 the sequential pair $\za\zb\ze^{-1}\zs\zg\zd$, where $u=\za\zb$, $w'=\ze^{-1}\zs$ and $v=\zg\zd$, but this is  not a sequential walk because the target of $\zs$ is a point on $u$ but not an endpoint of $u$. 
 
Our definition differs from \cite{HL} also because  sequential walks do not detect overlapping relations. For example,  the algebra given by the quiver
\[\xymatrix{1\ar[r]^\za&2\ar[r]^\zb&3\ar[r]^\zg&4}\]
bound by the relations $\za\zb=0=\zb\zg$ does not admit a sequential walk.
\end{remark}
\begin{remark}
 As we shall see below, if $A$ is a tree algebra of global dimension two, then the two notions of sequential walk and sequential pair coincide.
\end{remark}
\subsection{Why this notion is natural}
It is  well-known that if $A=\kk Q/I$ is an algebra, and $S_x$ a simple module, then $\pd \, S_x>1$ if and only if $x$ is the starting point of a top relation in $(Q,I)$. This is easily seen for instance by looking at the radical of $P_x$. The following considerations go in this direction.
\begin{lem}\label{lem 4.1}
 Let $A=kQ/I$ be an algebra and  $M$ an $A$-module of projective dimension $d>1$. Then there exists $x\in (\Supp M)_0$ and, for each $i\le d$, there exists $y_i\in Q_0$ such that $\Ext^i(S_x,S_{y_i})\ne 0$.
\end{lem}
\begin{proof}
 We first claim that $M$ has a composition factor $S_x$ such that $\pd\, S_x \ge d$. Consider the socle series 
 \[ 0\subsetneq \soc\, M\subsetneq \soc^2 M\subsetneq \cdots \subsetneq \soc^\ell M=M,
 \]
 where $\ell$ is the Loewy length of $M$, see for example \cite[V.I]{ARS}.
 If there exists a simple summand of $\soc\, M$ of projective dimension $d$, then we are done. Otherwise, there exists $i<\ell$ such that $\pd\, \soc^iM<d$ and $\pd\, \soc^{i+1} M =d$. The short exact sequence
 \[\xymatrix{0\ar[r]&\soc ^i M\ar[r]&\soc^{i+1} M\ar[r]&\frac{\soc^{i+1}M}{\soc^i M}\ar[r]&0}\]
 yields an exact sequence of functors
 \[\xymatrix{\Ext^d(\frac{\soc^{i+1}M}{\soc^i M}, -) \ar[r]& \Ext^d(\soc^{i+1} M, -)
 \ar[r]&\Ext^d(\soc^{i} M, -)=0.}\]
 Now $\Ext^d(\soc^{i+1} M, -)\ne 0$ implies $\Ext^d(\frac{\soc^{i+1}M}{\soc^i M}, -)\ne 0$. Hence, because the module $\frac{\soc^{i+1}M}{\soc^i M}$ is semisimple, there exists a simple composition factor $S_x$ of $M$ such that $\pd \, S_x\ge d$.
 
 We  deduce the result. Because of our claim, there exists a minimal projective resolution
 \[\xymatrix{P_{d+1}\ar[r]^{f_{d+1}}&P_{d}\ar[r]^{f_{d}}&P_{d-1}\ar[r]^{f_{d-1}}&\cdots\ar[r]^{f_1}&P_{0}\ar[r]^{f_{0}}&S_x\ar[r]&0}\]
 with $P_i\ne 0$ for all $i\le d$. Consider the projection $p_i\colon P_i\to \textup{top}\, P_i$. Then $p_i f_{i+1}=0$ because of the minimality of the resolution. Moreover, $p_i$ does not factor through $f_i$ because $P_i$ is the projective cover of $\textup{Ker} \,f_{i-1}$. Therefore $\Ext^i(S_x,\textup{top}\, P_i)\ne 0$ and there exists a simple composition factor $S_{y_i}$ of $\textup{top}\, P_i$ such that $\Ext^i(S_x,S_{y_i}) \ne 0$.
 \end{proof}
\begin{cor}
 \label{cor 4.2}
 Let $M$ be a module of projective dimension $d\ge 1$. Then there exists $x\in (\Supp M)_0$ which is the starting point of a top relation in $(Q,I)$.
\end{cor}
\begin{proof}
 Because $d\ge 2$, there exists $y_2$ such that $\Ext^2(S_x,S_{y_2})\ne 0$. We then apply \cite{Bo}.
\end{proof}

Let thus $M$ be a module with both projective and injective dimension larger than one. Because of Corollary \ref{cor 4.2} and its dual, there exist two points $x,y$ in $\Supp M$ which are respectively the starting point of a relation and the ending point of a relation. The notion of sequential walk (and all other similar notions like sequential pairs, double zeros etc.) arose from the attempt to connect $y$ to $x$ by a walk.

\subsection{A necessary condition for shod}  In this subsection, we prove our main result.

\begin{lem}
 \label{lem 331}
 Let $A$ have a sequential walk. Then there exists a reduction $B$ of $A$ containing one of the following (perhaps not full) subquivers.
 
 
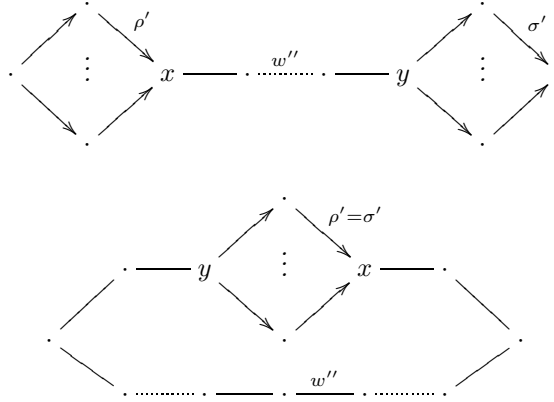
\[ \xymatrix@C20pt@R10pt { 
 & \cdot\ar[rd]^{\rho'}&&&&& \cdot\ar[rd]^{\zs'}
 \\
 \cdot\ar[ru]\ar[rd]&\vdots&x\ar@{-}[r]&\cdot\ar@{.}[r]^{w''}&\cdot\ar@{-}[r]& y\ar[ru]\ar[rd] &\vdots&\cdot
\\
&\cdot\ar[ru] &&&&&
\cdot\ar[ru]
}
 \]

 \[ \xymatrix@C20pt@R10pt { 
 && &\cdot\ar[rd]^{\rho'=\zs'}
 \\
 &\cdot\ar@{-}[r]&y\ar[ru]\ar[rd]&\vdots&x\ar@{-}[r]&\cdot\ar@{-}[rd]
 \\
\cdot\ar@{-}[ru]\ar@{-}[rd]&&&\cdot\ar[ru]&& &\cdot
\\
&\cdot\ar@{.}[r]&\cdot\ar@{-}[r]&\cdot\ar@{-}[r]^{w''}&\cdot\ar@{.}[r]&\cdot \ar@{-}[ru]\\
}
 \]
 where $\rho',\zs'$ are quadratic relations, $w''$ is a zigzag walk having no point in common with $\rho',\zs'$ except $x$ and $y$.
 
 Moreover, $w''$ generates a full subcategory of $B$.
\end{lem}
\begin{proof}
 Let $w=uw'v $ be a sequential walk in $A$. Then $u,v$ are branches of top relations $\rho=\sum\zl_iu_i$, $ \zs=\sum \mu_jv_j$ pointing to the same direction and $w'$ is a walk joining the ending point $x$ of $\rho$ to the starting point $y$ of $\zs$, and satisfying the condition of Definition \ref{main def}. We construct $B$ in two steps.
 
 (a) We eliminate excessive points. Let $e$ be the sum of the idempotents corresponding to: 
 \begin{enumerate}
\item The starting point and the ending point of each of $\rho$ and $\zs$.
\item All immediate successors of the starting points of each of $\rho$ and $\zs$ lying on one of the paths $u_i,v_j$. 
\item All sources and sinks of $w'$.
\end{enumerate}
Then the full subcategory $eAe$ of $A$ contains a subquiver of one of the forms shown in the statement. Notice however, that there may be arrows in $eAe$ between two points of the walk $w''$ which do not belong to $w''$ itself (for instance, multiple arrows) and also arrows between points of $w''$ and points lying on the $u_i',v_j'$.

(b) We eliminate exessive arrows. Consider the set $S$ of all arrows $\za$ in $eAe$ whose source and target lie in $w''$, but such that neither $\za$ nor $\za^{-1}$ belongs to $w''$. Because there are no relations in the zigzag walk $w''$, the set $S$ is a consistent cut. Let $E$ be the two-sided ideal of $eAe$ generated by $S$, and let $B=eAe/E$. Then we get an exact sequence
\[ \xymatrix{0\ar[r]&E\ar[r]&eAe\ar[r]&B\ar[r]&0}\]
realizing $eAe$ as a split extension of $B$ by $E$, as seen in Theorem \ref{thm 1.5}. Notice that, with this construction, $w''$ becomes a full subcategory of $B$.
\end{proof}

We call the reduction $B$ as in the  lemma the \emph{standard reduction} corresponding to a given sequential walk.

\begin{lem}
 \label{lem 332}
 Let $A$ have a sequential walk and  $B$ the corresponding standard reduction of $A$. Then the string module of the  zigzag walk $w''$ in $B$ has projective and injective dimension larger than one.
\end{lem}
\begin{proof}
 Let $M=M(w'')$ be the string module corresponding to $w''$, that is, $M$ is the $B$-module defined as a representation  by 
\[ M(x)=\left\{\begin{array}{ll}
\kk&\textup{if $x$ is a point of $w''$};\\
0&\textup{otherwise;}
\end{array}\right.\]
and
\[ M(\za)=\left\{\begin{array}{ll}
\textup{id}&\textup{if $\za$ is an arrow of $w''$};\\
0&\textup{otherwise.}
\end{array}\right.\]
Notice that in $B$, there may be arrows between points in $w''$ and points on one of the relation $\rho',\zs'$. But because of the definition of the string module, for any such arrow $\zb$, we have $M(\zb)=0$. Because no subpath of $w''$ is a branch of a relation in $B$, then $M$ is indeed a $B$-module. Moreover, $M$ is a string module, see \cite{BR}, and in particular, it is indecomposable. 

We now prove that $\textup{pd}\, M_B>1.$ The support of $M$ contains the starting point $y$ of $\zs'$, Therefore the projective cover of $M$ admits a direct summand $P_z$ such that $z$ lies on $w''$,  and either $z=y$ or  there is an arrow $z\to y$. It is easily seen that a nonprojective summand of $\zO^1 S_y$ is a direct summand of $\zO^1 M$. Because $\zO^1S_y$ is not projective, neither is $\zO^1M$. This establishes the claim.

Dually, we also have $\textup{id}\, M_B>1. $
\end{proof}

We are now ready to prove our main result. It generalizes \cite[2.4]{ARS}.

\begin{thm}
 \label{thm 3.3}
 Let $A=\kk Q/I$ be an algebra having a sequential walk. Then $A$ is not shod. In particular, $A$ is neither quasi-tilted, nor tilted.
\end{thm}
\begin{proof} 
Let $B$ be the standard reduction of $A$ corresponding to the sequential walk. 
By Lemma \ref{lem 332}, there exists an indecomposable $B$-module that has both injective and projective dimension larger than one. Therefore $B$ is not shod. 

Because of Corollary~\ref{cor 1.6}, neither is $A$.
\end{proof}

\begin{example}
 Consider the quiver
 \[\xymatrix@R30pt@C50pt{ 
 &&7\ar[lld]_(0.7){\za_1}\ar[ld]_(0.7){\za_2} \ar[rd]_(0.45){\za_3}\ar[rrd]_(0.45){\za_4}\ar[rrrd]^(0.7){\za_5}\\
 2\ar[rrd]_(0.3){\zb_1}&3\ar[rd]_(0.3){\zb_2}&&4\ar[ld]_(0.6){\zb_3}&5\ar[lld]_(0.6){\zb_4}&6\ar[llld]^(0.3){\zb_5}\\
 &&1
 }\]
 bound by the relations $\za_1\zb_1+\za_2\zb_2=0$,  $\za_3\zb_3+\za_4\zb_4+\za_5\zb_5=0$. We show how to perform the reduction procedure of Theorem \ref{thm 3.3}. Let here \[w=(\za_1\zb_1)\zb_3^{-1}\za_3^{-1} (\za_2\zb_2),\]
 with $u=\za_1\zb_1$, $v=\za_2\zb_2$ and $w'=\zb_3^{-1}\za_3^{-1}$.  It is clear that $u,v$ and $w'$ satisfy the conditions of Definition \ref{main def}. 
 
 We first eliminate points by taking $e=e_1+e_2+e_3+e_7$. Then $eAe$ is given by the quiver
 
 
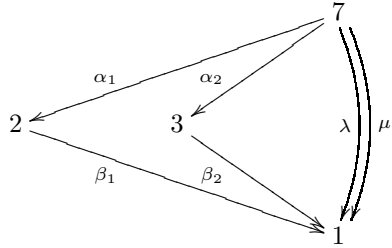
\[\xymatrix@R30pt@C50pt{ 
 &&7\ar[lld]_(0.7){\za_1}\ar[ld]_(0.7){\za_2} \ar@<2pt>@/^10pt/[dd]^\mu
  \ar@<-2pt>@/^10pt/[dd]_\zl\\
 2\ar[rrd]_(0.3){\zb_1}&3\ar[rd]_(0.3){\zb_2}\\ 
 &&1
 }\]
 bound by the relation $\za_1\zb_1+\za_2\zb_2=0$.
 This is a split extension of the algebra $B$ given by the quiver 
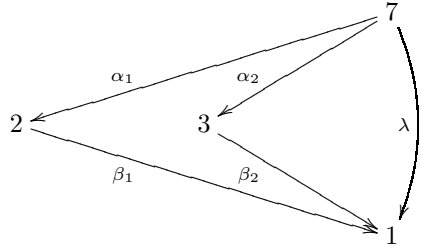
\[\xymatrix@R30pt@C60pt{ 
 &&7\ar[lld]_(0.7){\za_1}\ar[ld]_(0.7){\za_2} 
  \ar@/^10pt/[dd]_\zl\\
 2\ar[rrd]_(0.3){\zb_1}&3\ar[rd]_(0.3){\zb_2}\\ 
 &&1
 }\]
 bound by $\za_1\zb_1+\za_2\zb_2=0$, by the 
 two-sided ideal generated by the arrow $\mu$.
 
 The indecomposable $B$-module $M$ of the proof is the module $M=
\begin{smallmatrix}
 7\\1
\end{smallmatrix}$ 
supported by the arrow $\zl$.
Clearly, we have a minimal projective resolution
\[\xymatrix{0\ar[r]&P_1\ar[r]&P_2\oplus P_3\ar[r]&P_7\ar[r]&M\ar[r]&0}\]
so $\textup{pd}\,M_B=2$. Notice that $\zO^1 M=
\begin{smallmatrix}
 2\ 3\\1
\end{smallmatrix}
 $ while $\zO^1 S_7 =
\begin{smallmatrix}
 2\ 3\\1 
\end{smallmatrix}
 \oplus 1$. Thus $\zO^1M$ and $\zO^1 S_7$ have a common summand but are not equal.
 Similarly, $\id\,M_B=2.$ Thus $B$, and $A$,  are not shod.
\end{example}
\smallskip

The following examples illustrate that the converse of Theorem~\ref{thm 3.3} does not hold without additional conditions on the algebra $A$ or the module $M$.  In Section \ref{sect 4}, we  give examples of such additional conditions.
\begin{example}
 If an indecomposable module has both projective and injective dimension larger than one, this does not necessarily imply the existence of   a sequential walk. Indeed, there exist $x,y$ in the support of the module such that $x$ is the starting point of a top relation $v$, and $y$ the ending point of  a top relation $u$. Then there exists a walk $w'$ from $x$ to $y$ inside the support of $M$, but this does not imply that $w=uw'v$ is a sequential walk because it may not be reduced. For instance, let $A$ be the monomial tree algebra given by the quiver


 
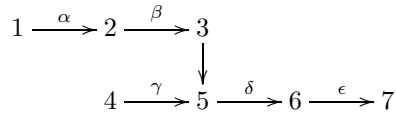
\[\xymatrix@R15pt{1\ar[r]^\za&2\ar[r]^\zb &3\ar[d]\\
 &4\ar[r]^\zg&5\ar[r]^\zd &6\ar[r]^\ze&7 }\]
 bound by the relations $\za\zb=0=\zg\zd\ze$,
then there is no sequential walk but the module $M=\begin{smallmatrix} 3\ 4\\5\end{smallmatrix}$ has both projective and injective dimension 2.
\end{example}
\begin{example}
 There exist sincere indecomposable modules of projective and injective dimension $2$. Let $A$ be given by the quiver
 \[\xymatrix{1\ar@<2pt>[r]^{\za_1}\ar@<-2pt>[r]_{\zb_1}
 &2 \ar@<2pt>[r]^{\za_2}\ar@<-2pt>[r]_{\zb_2}
&3}\]
bound by the relations $\za_1\za_2=0$ and $\zb_1\zb_2=0$. Note that $A$ is gentle and in particular tame. The indecomposable module $M=\begin{smallmatrix}1\\2\\3\end{smallmatrix}$ is sincere and of both projective and injective dimension 2.
\end{example}

\section{Applications and examples}\label{sect 4}
\subsection{The case of global dimension two}
Because sequential walks do not detect overlaps, it is natural to think of them in the context of algebras of global dimension two.
\begin{prop}\label{prop 4.1}
 If $A$ is a monomial algebra of global dimension 2 and $M$ is a uniserial $A$-module whose injective and projective dimensions are both larger than one then there exists a sequential walk in $A$.
\end{prop}
\begin{proof}
  Since $M$ is uniserial, the support of $M$ is of the form 
  \[\xymatrix{z_1\ar[r]& z_2\ar[r]&\cdots& \ar[r]& z_\ell}.\] By Corollary \ref{cor 4.2} and its dual, there exist $z_i,z_j\in\Supp M$ such that $z_j$ is the starting point of a top relation $v$, and $z_i$ the ending point of  a top relation $u$. Since $A$ is monomial of global dimension 2, the two relations $u, v$ do not overlap \cite{GHZ}. Thus $i\le j$. Now let $w'$ be a path $\xymatrix@C12pt{z_i\ar[r]&\cdots\ar[r]&z_j}$ in $\Supp M$.
  Then the composition $uw'v$ is a sequential walk in $A$.
\end{proof}
\begin{cor}\label{cor 4.2b}
 Let $A$ be a Nakayama algebra of global dimension 2. Then there exists a sequential walk in $A$ if and only if $A$ is not tilted. 
\end{cor}
\begin{proof}
 Necessity follows from Theorem \ref{thm 3.3}. To show sufficiency, suppose $A$ is not quasi-tilted. Then there exists an indecomposable $A$-module $M$ of both projective and injective dimension 2. Since $A$ is a Nakayama algebra, $A$ is monomial and $M$ is uniserial. Now the result follows from Proposition \ref{prop 4.1} and the fact that quasi-tilted Nakayama algebras are representation-finite, and hence tilted.
\end{proof}

We have the following characterization of projective dimension 2.
\begin{prop}
 \label{prop 4.3}
Let $M$    be an indecomposable module over an algebra of global dimension 2 such that one of the sinks in the support of $M$ is the starting point of a top relation. Then the projective dimension of $M$ is two.
\end{prop}
\begin{proof}
 Indeed, we have $\pd \, S_x>1$ and an exact sequence 
 \[\xymatrix{0\ar[r]&S_x\ar[r]&M\ar[r]&M/S_x\ar[r]&0.}\]
 Hence we have an exact sequence of functors
 \[\xymatrix{\Ext^2(M/S_x,-) \ar[r]&\Ext^2(M,-)
 \ar[r]&\Ext^2(S_x,-)
 \ar[r]&0,}\]
 because $\Ext^3=0$. Hence $\Ext^2(S_x,-)\ne 0$ implies $\Ext^2(M,-)\ne 0$.
\end{proof}

\smallskip

For tree algebras of global dimension two, sequential walks are easy to characterize.
\begin{prop}
 Let $A$ be a tree algebra of global dimension two. Then the two notions of sequential walk and sequential pair coincide.
\end{prop}
\begin{proof}
 Indeed, because $A$ is a tree algebra (not necessarily a string algebra) it is monomial, and the global dimension two means that no two top relations overlap. Because $Q$ is a tree, two  points of $Q$ are connected by a unique walk, and so a sequential walk as well as a sequential pair mean a walk of the form $w=uw'v$, where $u,v$ are monomial relations pointing to the same direction while $w'$ is any walk not containing a relation.
\end{proof}
\subsection{An application to laura algebras}
For our next corollary, we  recall a few notions. Given an algebra $A$, we denote by $\textup{ind}\, A$ a full subcategory of $\textup{mod}\,A$ consisting of exactly one representative from each isomorphism class of indecomposable modules. The \emph{left part} $\mathcal{L}_A$ of $\textup{mod}\,A$ consists of all modules $M$ in $\textup{ind}\, A$ such that, for every $L$ for which there exists a path of nonzero morphisms from $L$ to $M$, we have $\pd\,L\le 1$. The \emph{right part} $\mathcal{R}_A$ of  $\textup{mod}\,A$ is defined dually, and an algebra $A$ is called a \emph{laura algebra} if and only if $\mathcal{L}_A\cup\mathcal{R}_A$ is cofinite in $\textup{ind}\, A$.

In \cite{Dionne}, Dionne  has shown that a string algebra is laura if and only if its bound quiver does not contain a combinatorial configuration called intertwined double zero which we now define.
 
\begin{definition}\cite[2.1.1]{Dionne}\label{def 44} Let $A=\kk Q/I$ be a string algebra. A reduced walk $w$ in $Q$ is called an \emph{intertwined double zero} if $w=\rho_1w_1w_2w_3\rho_2$ where 
\begin{itemize}
\item [{\rm (a)}] $\rho_1=\za_1\ldots\za_n, \rho_2=\zb_1\ldots\zb_m$ are monomial relations pointing in the same direction,
\item [{\rm (b)}] neither $\za_2\ldots\za_n w_1w_2w_3\zb_1\ldots\zb_{m-1}$ nor its inverse contains a monomial relation, and
\item [{\rm (c)}] $w_2$ is a band.
\end{itemize}

The next corollary shows that one direction of Dionne's result follows from ours. Observe first that, if $A=\kk Q/I $ is a string algebra, and $w$ is an intertwined double zero then $w$ is a sequential walk in our sense.
\begin{cor}
 Let $A$ be a string algebra having an intertwined double zero. Then $A$ is not laura.
\end{cor}
 
\begin{proof}
 It follows from  Definition \ref{def 44} that if $w=\rho_1w_1w_2w_3\rho_2$ as above is an intertwined double zero, then, for any $n\ge 1$, the reduced walk $w_n=\rho_1w_1w_2^nw_3\rho_2$ is also a sequential walk in the bound quiver of $A$. Because of Lemma \ref{lem 332}, we have $\pd\,M(w_n)>1$ and $\id\,M(w_n)>1$. In particular, the $M(w_n)$ form an infinite family of nonisomorphic indecomposable modules lying neither in $\mathcal{L}_A$ nor in $\mathcal{R}_A$. Thus $A$ is not laura.
\end{proof}
\end{definition}
 
\subsection{An application to 2-Calabi-Yau tilted algebras}
Let  $C=\kk Q/I$ be a quasi-tilted algebra. In particular, its relation extension $\widetilde C =\kk \widetilde Q/\widetilde I$, which is the trivial extension of $C$ by the bimodule $\textup{Ext}^2_C(DC,C)$, is a cluster-tilted algebra or a 2-Calabi-Yau tilted algebra of canonical type, see \cite[3.1]{ASS}. A walk $w=\za w'\zb$ in $(\widetilde Q,\widetilde I)$ is called a \emph{$C$-sequential walk}
if\begin{itemize}
\item [(i)] $w'$ consists entirely of old arrows, see \cite{ABS} for the terminology ``old'' vs ``new'' arrows;
\item[(ii)] $\za,\zb$ are new arrows corresponding respectively to old relations $\rho=\sum_i\zl_iu_i$ and  $\zs=\sum_j\mu_jv_j$;
\item[(iii)] For all $i,j$, the walk $w=u_iw'v_j$ is sequential in $(Q,I)$.
\end{itemize}
Then we have
\begin{cor}
 \label{cor 3.5}
 Let $C$ be a quasi-tilted algebra. Then the bound quiver of its relation extension $\widetilde C$ contains no $C$-sequential walk.
\end{cor}

\subsection{Example}

\label{ex 1bis} 
We have seen in Example \ref{ex 1} that
the algebra $A$ given by the quiver
 \[\xymatrix{1\ar@<2pt>[r]^{\za_1}\ar@<-2pt>[r]_{\zb_1}
 &2 \ar@<2pt>[r]^{\za_2}\ar@<-2pt>[r]_{\zb_2}
&3}\]
bound by the relation $\za_1\za_2+\zb_1\zb_2 =0$ is a tilted algebra. Notice that here $w=(\za_1\za_2)\zb_2^{-1}\za_1^{-1}(\zb_1\zb_2)$ is not a sequential walk in the sense of Definition \ref{main def}, since the subpath $(w')^{-1}=\za_1\zb_2$ has arrows in common with the branches of the relation $u=\za_1\za_2$ and $v=\zb_1\zb_2$. In fact, this bound quiver contains no sequential walks.

On the other hand, the algebra $A'$ given by the same quiver but bound by the relation $\za_1\za_2=0$ contains evidently the sequential walk $w=(\za_1\za_2)\zb_2^{-1}\zb_1^{-1}(\za_1\za_2)$. In particular, Theorem \ref{thm 3.3} implies that this algebra is not tilted.

It is interesting to note that the relation extensions of both algebras $A$ and $A'$ have the same quiver $\widetilde Q$. Therefore the associated cluster categories $\mathcal{C}_A$ and $\mathcal{C}_{A'}$ are categorifications of the same cluster algebra $\mathcal{A}(\widetilde Q)$.

\end{document}